\documentclass[12pt]{article}
\usepackage{amsmath,amsthm,amssymb,amsfonts}
\usepackage{hyperref}

\usepackage{amsmath,amssymb,amsthm,mathrsfs,geometry}

\usepackage{geometry}
\newlength{\NotationsTextwidth}
\newtheorem{theorem}{Theorem}[section]
\newtheorem{definition}[theorem]{Definition}
\newtheorem{lemma}[theorem]{Lemma}
\newtheorem{proposition}[theorem]{Proposition}
\newtheorem{corollary}[theorem]{Corollary}
\theoremstyle{remark}
\newtheorem{remark}[theorem]{Remark}

\newcommand{\myparagraph}[1]{} 

\title{On the Problem of Mixed-Tateness of the Motives of $G$-Varieties}
\author{\myparagraph{Esmail Arasteh Rad and }Somayeh Habibi}
\date{\today}

\begin{document}
\maketitle
\begin{abstract}

Building on earlier work concerning the motives of $G$-bundles, we study the structure of motives associated with certain classes of $G$-varieties. In particular, we show that the corresponding motives lie within the category of mixed-Tate motives, under certain condition on the stabilizers. We further discuss some applications and provide some examples to illustrate the limitations.   
\end{abstract}

\section{Introduction}

Mixed-Tateness is a strong motivic property with significant arithmetic and geometric consequences. 
For example, over finite fields, varieties with mixed Tate motives have counting functions given by 
polynomials in the field size, linking them to $\mathbb{F}_1$-geometry and 
combinatorial models such as toric varieties \cite{Soule2004,Deitmar2005}. 
On the motivic side, mixed Tate objects form a tannakian subcategory with relatively simple Galois 
groups, making their structure more accessible and allowing explicit computations 
\cite{Brown2012}. They thus serve as a natural testing ground for conjectures in the 
theory of motives and periods, e.g. see \cite{Brosnan2005} and  \cite{ConnesConsaniMarcolli2008}. Mixed Tate motives also play a remarkable role from quantum field theory point of view. It was originally conjectured that the residues of Feynman integrals in scalar field theories are always periods of mixed Tate motives. This conjecture follows from another conjecture by Kontsevich, which proposed that graph hypersurfaces are mixed Tate; see \cite{Marcolli2010}. However, it was later observed that certain Feynman graphs can give rise to K3 surfaces \cite{BrownSchnetz2012}. Nevertheless, it remains an important problem to classify the graphs whose corresponding hypersurface has mixed Tate property. They also appear in a model of gravity based on the spectral action functional, see \cite{Fa-Ma}.

As we mentioned above, the question of when algebraic varieties have mixed-Tate motives appears in diverse settings, including configuration spaces, toric varieties, moduli spaces of bundles, and many varieties arising naturally in the representation theory of algebraic groups. In the latter the favorable symmetry and the corresponding combinatorial structures, together with the properties of stabilizers, govern the outcome.

Using De Concini-Procesi wonderful compactification \cite{DeConciniProcesi1983}, and motivic Leray-Hirsch, the author produced a nested filtration on the motive $M(\mathcal G)$, of a $G$-bundle $\mathcal G$ over $X$, see \cite{AH17} or  \cite{AH25} and \cite{MyThesis}. Then, in \cite{H-R}, using Huber and Kahn slice filtration \cite{H-K}, a two sided implication is proved, showing that the motive of $\mathcal G$ is mixed Tate if and only if $X$ is mixed Tate, under certain circumstances. As an application of these results, in this note, we discuss the mixed Tateness of the motive of certain $G$-varieties.

\section*{Acknowledgment:}
I warmly thank B. Kahn for useful comments, and E. Arasteh
 Rad for useful comments and inspiring discussions. I am grateful to L.
 Barbieri Viale for steady encouragement.\\
 This research was in part supported by a grant from IPM (No.1402140032).

\tableofcontents

\section{Notation}

To denote the motivic categories over $k$, such as 

$$
\textbf{DM}_{gm}(k),~ \textbf{DM}_{gm}^{eff}(k),~ \textbf{DM}_{-}^{eff}(k),~\text{etc.} 
$$  
and the functors $M:\textbf{Sch}_k\rightarrow \textbf{DM}_{gm}^{eff}(k)$ and $M^c:\textbf{Sch}_k\rightarrow \textbf{DM}_{gm}^{eff}(k)$ we use the same notation that was introduced  in \cite{VSF}. 
When $k$ is of positive characteristic, we assume coefficients are in $\mathbb{Q}$. 
\\

For the definition of the geometric motives with compact support in positive characteristic we also refer to \cite[Appendix B]{H-K}.\\

\section{Motive of Rationally Special $G$-varieties}

Let us first recall the definition of the category of mixed-Tate motives:

\begin{definition}
We denote by $TDM^{eff}_{gm} (k)$ the thick tensor subcategory of $DM^{eff}_{gm} (k)$,
generated by $\mathbb{Z}(0)$ and the Tate object $\mathbb{Z}(1)$. An object of  $TDM^{eff}_{gm} (k)$ is called a mixed Tate motive.  $TDM^{eff}_-(k)$ is the localizing subcategory of  $DM^{eff}_- (k)$
generated by  $TDM^{eff}_{gm} (k)$.
\end{definition}

Let $G$ be a connected split reductive algebraic group over an algebraically closed field $k$.
Let us first recall the following result:

\begin{theorem}\label{G-bundles, two sided statement}
Let $X$ be a principal $G$-bundle over $Y$. Furthermore, assume $X$ is locally trivial for the Zariski topology on $Y$.  If $Y$ is smooth and $M(X)$ is mixed Tate, then $M(Y)$ is mixed Tate. Moreover if $M(Y)$ is mixed Tate, then $M(X)$ is mixed Tate.
\end{theorem}

\begin{proof}
  This is a particular case of \cite[Thm. 0.1]{H-R}.
\end{proof}

\begin{corollary}\label{CorGismixedTate}
 The motive $M(G)$ of a connected split reductive group $G$ is mixed Tate.
\end{corollary}

\begin{proof} Obvious.

\end{proof}
  
 \begin{proposition}
     The motive associated to the variety $V_+(det(x_{ij}))$ in $\mathbb{P}^{n^2-1}$ is mixed Tate.
 \end{proposition}
\begin{proof}
By the Corollary \ref{CorGismixedTate}, $M(GL_n)$ is mixed Tate. Let $X_r$ denote the subvariety consisting of rank $r$ matrices in $V(det(x_{ij})) \subseteq \mathbb{A}^{n^2}$. The variety $X_r$ can be realized as the quotient of $GL_n$ by the stabilizer of its action on $X_r$,which is isomorphic to $ GL_r \times  \mathrm{U} $, where $\mathrm{U}$ is the unipotent subgroup. Consider the map $GL_n \to GL_n/GL_r$, now mixed Tateness of $GL_n$ together with Hilbert Satz90 and the theorem \ref{G-bundles, two sided statement} implies that $M(GL_n/GL_r)$ is mixed Tate. As $GL_n/GL_r$ is affine bundle over $X_r$, we argue that $M(X_r)$ is mixed Tate. Thus again by theorem \ref{G-bundles, two sided statement} and Hilbert Satz90, the motive of the subvariety $\Tilde{X_r}=(X_r-\{0\}) / \mathbb{G}_m$  in $V_+(det(x_{ij}))$ is mixed Tate. We argue that $V_+(det(x_{ij}))$ is stratified mixed Tate and the Proposition follows from \cite{H-R}, Lemma 2.6.

\end{proof}

\begin{definition}

    \begin{enumerate}
        \item We say that the pair $(G,H)$, with $G$ connected reductive and $H \subset G$ a connected closed subgroup, is a rationally special pair if the class of $H$-torsor $G\to G/H$ maps to the trivial element in $H^1(K,H)$, under the obvious map
        $H^1(G/H,H)\to H^1(K,H)$. Here $K$ denotes the function field of the homogeneous space $G/H$.
        \item Let $X$ be a quasi-projective $G$-variety. We say that $X$ is a rationally special $G$-variety if for all points $x$ of $X$ the pair $(G,stab_x)$ is a rationally special pair.

    \end{enumerate}

\end{definition}

\begin{remark}
    Let $K = k(G/H)$ be the function field of the variety $G/H$. Note that for a reductive group $G$ and a closed subgroup $H$ the class of the generic fiber $\pi^{-1}(\eta)$ of $\pi:G \to G/H$, in the cohomology set $H^1(K, H)$ might be trivial, however, the full cohomology set $H^1(K, H)$ is nontrivial. As an example, set $G = \mathrm{SO}(4)$, $H = \mathrm{SO}(3)$. Furthermore, assume that $char~k$ is not equal to $2$ and $k$ contains square root of $-1$.
We embed $H$ into $G$ as the subgroup fixing the first coordinate:
\[
H \hookrightarrow G, \quad B \mapsto \begin{pmatrix} 1 & 0 \\ 0 & B \end{pmatrix}.
\]
Then $H$ is a closed subgroup of $G$, and the quotient variety $G/H$ is isomorphic to the $3$-dimensional sphere:
\[
S^3 = \{ (x,y,z,w) \in \mathbb{A}^4 : x^2 + y^2 + z^2 + w^2 = 1 \}.
\]
Indeed, $G$ acts transitively on $S^3$, and the stabilizer of $(1,0,0,0)$ is exactly $H$.
 
 As $S^3$ can be identified with the group of unit quaternions over $k$, its tangent bundle is parallelizable. Since the $SO(3)$-bundle $\pi$ can be identified with the space of orthonormal frames of the tangent bundle, we see that the corresponding class in the cohomology set $H^1(K, \mathrm{SO}(3))$ is trivial. 

On the other hand, the cohomology set $H^1(K, \mathrm{SO}(3))$ classifies $3$-dimensional nondegenerate quadratic forms over $K$ of trivial discriminant, up to isometry. Since $K$ is a purely transcendental extension of $k$ of transcendence degree $3$, there exist many anisotropic $3$-dimensional quadratic forms of trivial discriminant. For example, consider
\[
q(x,y,z) = x^2 + t y^2 + t z^2, \qquad t \in K^\times \text{ not a square}.
\]
Its Gram matrix is $\operatorname{diag}(1,t,t)$, so
\[
\det(q) = t^2,\quad
\mathrm{disc}(q) = (-1)^{3} \det(q) \equiv -t^2 \equiv -1 \in K^\times/(K^\times)^2.
\]
Since $k$ is algebraically closed, $-1$ is a square in $K$, hence the discriminant of $q$ is trivial. As $t$ is not a square in $K$, this quadratic form is not isometric to the split form $\langle 1,1,1 \rangle$. Thus $q$ defines a nontrivial element of $H^1(K, \mathrm{SO}(3))$.

\end{remark}

\begin{theorem}\label{thmmain}
 Let $X$ be a rationally special $G$-variety. Then the motive $M(X)$ associated with $X$ in the Voevodsky's category of geometric motives $DM_{gm}(k)$ is mixed Tate.  
\end{theorem}

\begin{proof}
    The $G$-action defines a stratification $\{X_i\}_{i\in I}$ of $X$ to the $G$-orbits. Now, we have $X_i\cong G/stab_{x_i}$ for a closed point $x_i$ in $X_i$. As the pair $(G,H_i:=stab_{x_i})$ is special, the $H_i$-torsor $\pi_i:G \to G/H_i$ admits a rational section defined over the function field $K_i$ of $X_i$, which extends to a regular section over some nonempty Zariski open $U_i \subset X_i$. Over this open set we have
    \[
  \pi_i^{-1}(U_i) \cong U_i \times H_i.
  \]
Translating $U_i$ by the action of $G$ yields a Zariski open covering by such trivializations. Hence, we may argue that $X_i$ is mixed Tate by theorem \ref{G-bundles, two sided statement} and corollary \ref{CorGismixedTate}. Now since $X$ admits an stratification by mixed Tate varieties, we observe that $M(X)$ is mixed Tate, see \cite[Lem. 2.6]{H-R}.

\end{proof}

\begin{remark}
    Note in particular that, according to the theorem \ref{thmmain}, for a rationally special pair $(G,H)$, the motive of the quotient variety $G/H$ is mixed Tate.
\end{remark}

\begin{remark}

Note that, since for a rationally special pair $(G,H)$ the quotient map $\pi:G \to G/H$ is Zariski locally trivial with fiber $H$, we have the equality $[G] = [G/H]\cdot[H]$ in the Grothendieck ring of varieties $K_0(\mathrm{Var}_k)$.
\end{remark}

\begin{proposition}
The $SO_n$-torsor 
\[
SL_n \;\longrightarrow\; SL_n/SO_n
\]
is not Zariski-locally trivial for $n\ge 3$ over a field $k$ with $\operatorname{char}k\neq 2$. Moreover, the class corresponding to the generic fiber in $H^1(K,SO_n)$ is not trivial, here $K$ is the function field of $SL_n/SO_n$. 
\end{proposition}

\begin{proof}
Let \(X := SL_n/SO_n.\) One may identify \(X\) with symmetric matrices \(Sym_n\) with \(\det S = 1 \). Thus the corresponding coordinate ring is
\[
k[X] = k[s_{ij}\ (1\le i\le j\le n)]/(\det(s_{ij})-1).
\]
The matrix $S=(s_{ij})$ with entries the coordinate functions is called the
\emph{universal symmetric matrix}.
Define the \emph{universal quadratic form}
\[
q_{\mathrm{univ}}(v) = v^T S v, \qquad v\in k[X]^n,
\]
and after base change to the function field $k(X)$, this gives
\[
q_{\mathrm{univ}}\in \mathrm{Quad}_n(k(X)).
\]

The projection $SL_n\to X$ is a principal $SO_n$-bundle (torsor).  
Its restriction to the generic point $Spec~k(X)\to X$ gives a class
\[
[q_{\mathrm{univ}}] \in H^1(k(X),SO_n),
\]
which corresponds to the quadratic form $q_{\mathrm{univ}}$.
If the torsor $SL_n\to X$ were Zariski-locally trivial, then its class in
$H^1(k(X),SO_n)$ would be trivial, i.e.\ $q_{\mathrm{univ}}$ would be
isometric to the split quadratic form of rank $n$ over $k(X)$. To see that this is not the case, consider the even Clifford algebra
$C_0(q_{\mathrm{univ}})$, whose Brauer class
\[
\gamma := [C_0(q_{\mathrm{univ}})] \in \operatorname{Br}(k(X))
\]
is an invariant of the isometry class of $q_{\mathrm{univ}}$.  
If $q_{\mathrm{univ}}$ were split then $\gamma=0$. Let's compute a specialization. Define
\[
S_{a,b} = \operatorname{diag}\!\big(a,\;b,\;(ab)^{-1},\;1,\dots,1\big)
\in Sym_n\big(k(a,b)\big).
\]
This has determinant $1$, hence defines a rational point
$Spec~k(a,b)\to X$.  
The specialization of $\gamma$ to $\operatorname{Br}(k(a,b))$ is the
Brauer class of the form
\[
q_{a,b} = \langle a,\ b,\ (ab)^{-1},\ 1,\dots,1\rangle.
\]

Now consider the ternary subform
\[
q_3 = \langle a,\ b,\ (ab)^{-1}\rangle.
\]
A direct computation of Clifford algebras for diagonal forms shows that
\[
C_0(q_3) \;\cong\; (a,b),
\]
the quaternion algebra over $k(a,b)$ with generators
$i,j$ subject to $i^2=a$, $j^2=b$, $ij=-ji$.  
Thus the specialization of $\gamma$ contains the quaternion class $(a,b)$.

Finally, the quaternion algebra $(a,b)$ over $k(a,b)$ is nontrivial in
$\operatorname{Br}(k(a,b))$.  
Indeed, after embedding $k(a,b)\hookrightarrow k((a))((b))$,
the algebra $(a,b)$ remains a division algebra over the iterated Laurent
series field, hence represents a nonzero Brauer class.  
Therefore $\operatorname{res}(\gamma)\neq 0$, and so $\gamma\neq 0$ in
$\operatorname{Br}(k(X))$.

Consequently, $q_{\mathrm{univ}}$ is not isometric to the split form over
$k(X)$, so the torsor class in $H^1(k(X),SO_n)$ is nontrivial.
It follows that the principal $SO_n$-bundle $SL_n\to SL_n/SO_n$ is not
Zariski-locally trivial.
\end{proof}

As we have seen above the quotient map $SL_n\to SL_n/SO_n$ is not Zariski locally trivial. In addition, as the counting formula indicates, see remark \ref{remCount} below, the motive $M(Sym_n^\times)$ is not mixed Tate in $DM_{gm}(k)$, however, as it will be shown in the proposition \ref{propSym_QisMT}, this is still the case after passing to rational coefficients.

Fix an \(n\)-dimensional vector space \(V\) over a field \(k\). 
The group \(GL_n(k)\) acts by congruence on the space of symmetric, nondegenerate bilinear forms on \(V\). 
If we fix one such form \(Q_0\), its stabilizer is the orthogonal group \(O_n\) associated to \(Q_0\). 
Hence, as a homogeneous space, \(GL_n/O_n \) parametrizes nondegenerate symmetric bilinear forms on $V$ in the $GL_n$-orbit of $Q_0$.
In other words, \(GL_n/O_n\) parametrizes nondegenerate symmetric bilinear (or quadratic, if \(\mathrm{char}(k)\neq 2\)) forms of the same isometry type as the chosen form. 
Over finite fields of odd order, there is only one such orbit if \(n\) is odd, and two such orbits if \(n\) is even.

\begin{remark}\label{remCount}
\begin{enumerate}
\item[a)] 
(Odd dimension \(n=2m+1\).)
In this case there is a single isometry class of nondegenerate quadratic forms. The order of the orthogonal group is
\[
|O_{2m+1}(\mathbb{F}_q)| 
= 2\,q^{m^2}\prod_{i=1}^m \bigl(q^{2i}-1\bigr).
\]
Therefore
\[
\#(GL_{2m+1}/O_{2m+1})(\mathbb{F}_q) 
= \frac{\displaystyle\prod_{i=0}^{2m}(q^{2m+1}-q^i)}{2\,q^{m^2}\prod_{i=1}^m (q^{2i}-1)}.
\]
\item[b)] (Even dimension \(n=2m\).)  
For \(n\) even there are two isometry classes of nondegenerate quadratic forms, denoted by signs \(+\) (split) and \(-\) (nonsplit). The orders of the corresponding orthogonal groups are
\[
|O_{2m}^{\pm}(\mathbb{F}_q)| 
= 2\,q^{m(m-1)}(q^m \mp 1)\prod_{i=1}^{m-1} \bigl(q^{2i}-1\bigr).
\]
Hence
\[
\#(GL_{2m}/O_{2m}^{\pm})(\mathbb{F}_q)
= \frac{\displaystyle\prod_{i=0}^{2m-1}(q^{2m}-q^i)}{2\,q^{m(m-1)}(q^m \mp 1)\prod_{i=1}^{m-1}(q^{2i}-1)}.
\]
\end{enumerate}
\end{remark}

\noindent
Let us recall the following definition

\begin{definition}
Let $f : X \to Y$ be a proper surjective morphism of algebraic varieties.  
For each integer $k \ge 0$, set
\[
  Y^k \;=\; \{ y \in Y \;\mid\; \dim f^{-1}(y) \ge k \}.
\]
The map $f$ is called \emph{semi-small} if
\[
  \dim Y^k \;\le\; \dim X - 2k \qquad \text{for all } k \ge 0.
\]
\end{definition}

\begin{proposition}\label{propSym_QisMT}
    The motive $M(Sym_n^\times)_\mathbb Q$ is mixed Tate in $DM_{gm}(k)_\mathbb Q$.
\end{proposition}

\begin{proof}
Let \(V\) be an \(n\)-dimensional \(k\)-vector space, \(n\ge2\). \(W=Sym_2 V^\vee\) and \(N=\dim W=\dfrac{n(n+1)}2\). We identify \(\mathbb P(W)\cong\mathbb P^{N-1}\). Let \(Y\subset\mathbb P(W)\) denote the determinantal hypersurface \(\{\det=0\}\). By theorem \ref{G-bundles, two sided statement} and localizing exact sequence it is enough to show that $M(Y)$ is mixed Tate. Set \(\mathcal I=\{(Q,[v])\in\mathbb P(W)\times\mathbb P(V)\mid Qv=0\}\) with projections \(p:\mathcal I\to\mathbb P(V)\), \(\pi:\mathcal I\to\mathbb P(W)\).

The projection \(p: \mathcal I\to\mathbb P(V)\cong\mathbb P^{\,n-1}\) exhibits \(\mathcal I\) as a projective bundle:
			\[
			\mathcal I \cong \mathbb P(\mathcal E)
			\]
			for a rank \((N-n)\) vector bundle \(\mathcal E\) on \(\mathbb P^{n-1}\). Consequently \(\mathcal I\) is smooth and
			\[
			M(\mathcal I)_\mathbb Q \cong \bigoplus_{j=0}^{N-n-1} M(\mathbb P^{n-1})_\mathbb Q(j)[2j]
			.\]
It is clear that the map $\pi: \mathcal I\to \mathbb P(W)$ factors through $Y$ and gives a resolution for the singularities of $Y$. 

\begin{lemma}
Let $V$ be an $n$--dimensional vector space over an algebraically closed field,
and let $Y \subset \mathbb P(Sym_2 V^\vee)$ be the determinantal hypersurface
of singular symmetric forms. Consider the incidence variety
\[
  \mathcal I \;=\; \{ ([Q],[\ell]) \in \mathbb P(Sym_2 V^\vee) \times \mathbb P(V) \;\mid\; 
  \ell \subset \ker(Q)\},
\]
with projection $\pi:\mathcal I\to Y$. Then $\pi$ is a semi--small map.
\end{lemma}

\begin{proof}

Fix $r$, $0\le r\le n-1$, and consider the stratum $S_r\subset Y$ consisting of forms of rank $r$.
If $Q$ has rank $r$, then $\dim\ker(Q)=n-r$, and the fibre
\[
  \pi^{-1}([Q]) \;\cong\; \mathbb P(\ker Q) \;\cong\; \mathbb P^{n-r-1}
\]
has dimension $s_r=n-r-1$. The dimension of the stratum $S_r$ is
\[
  d_r \;=\; r(n-r) \;+\; \Big(\tfrac{r(r+1)}2-1\Big)
  \;=\; rn - \tfrac{r(r-1)}2 - 1.
\]
Namely, for a symmetric form $Q$ of rank $r$,the support (or the nondegenerate part) of $Q$ is an $r$-dimensional subspace $U$ of $V$ on which $Q$ restricts to a nondegenerate symmetric form.  Choosing the $r$-dimensional subspace $U$ of $V$ is a point of the Grassmannian $Gr(r,V)$, which has dimension $r(n-r)$. In addition on a fixed 
$r$-plane $U$, symmetric bilinear forms are parametrized by $Sym_2U^\vee$, which has vector-space dimension $r(r+1)/2$, then passing to projective classes. The total dimension of $\mathcal I$ is
\[
  \dim \mathcal I \;=\; \dim \mathbb P(V) + \dim \mathbb P(Sym_2 V^\vee) - n
  \;=\; (n-1) + \Big(\tfrac{n(n+1)}2 -1\Big) - n
  \;=\; \tfrac{n(n+1)}2 - 2.
\]
Hence the semi--small inequality for $S_r$ is
\[
  d_r + 2s_r \;\le\; \dim \mathcal I.
\]
Subtracting the left from the right we obtain
\[
  \Delta(r) \;:=\; \dim \mathcal I - (d_r+2s_r)
  \;=\; \tfrac{(n-r-1)(n-r-2)}{2}.
\]
Since $\Delta(r)\ge 0$ for all $0\le r\le n-1$, the inequality holds for every stratum. Thus $\pi$ is semi--small.
\end{proof}

Now, we apply the motivic decomposition theorem of Migliorini and de Cataldo for semi-small morphisms, see \cite[Theorem 4.0.4]{Cat-Mig}, to the morphism $\mathcal I\to Y$. 
Putting $\Delta(r)=0$ we observe that the map $\pi$ is semi--small, and the only relevant strata are 
the corank $1$ stratum $S_{n-1}$ and the corank $2$ stratum $S_{n-2}$. For $S_{n-1}$, one has $\pi^{-1}([Q])\cong \mathbb P^0$, hence 
$t_{n-1}=\tfrac12(\dim \mathcal I-\dim S_{n-1})=0$. The set of irreducible components of the fibre is a singleton, so the associated 
covering $Z_{n-1}\to S_{n-1}$ is trivial. One may take 
$\overline{Z}_{n-1}=Y$ as a projective compactification. For $S_{n-2}$, one has $\pi^{-1}([Q])\cong \mathbb P^1$, hence 
$t_{n-2}=\tfrac12(\dim \mathcal I-\dim S_{n-2})=1$.
Again the fibre is irreducible and the covering is trivial. Therefore, by the embedding theorem \cite[Prop. 2.1.4]{VSF}, we get the motivic splitting
\[
   M(\mathcal I)_{\mathbb Q} \;\;\cong\;\; M(Y)_{\mathbb Q} \;\oplus\; M(\overline S_{n-2})_{\mathbb Q}(1)[2],
\]
Consequently, the motive $M(Y)$ is mixed-Tate.

\end{proof}

\section{Consequences}

\subsection{Homogeneous space of an irreducible representation of $\mathrm{SL}_n$}

Irreducible algebraic representations of $SL_n$ are classified by dominant integral highest weights 
\(\lambda = a_1\omega_1 + a_2\omega_2 + \cdots + a_{n-1}\omega_{n-1}\), where the $\omega_i$ are the fundamental weights and $a_i \in \mathbb{Z}_{\ge 0}$. 
Each highest weight $\lambda$ determines a finite-dimensional irreducible representation, realized functorially via the Schur functor associated to the Young diagram of $\lambda$. 
This yields an embedding 
\[
\rho_\lambda : SL_n \hookrightarrow GL_{N_\lambda}, \qquad N_\lambda = \dim V_\lambda,
\]
whose image acts irreducibly on $\mathbb{C}^{N_\lambda}$ and is therefore not contained in any proper parabolic subgroup of $GL_{N_\lambda}$. We denote by $\mathcal X(n,\lambda)$ the quotient variety $GL_{N_\lambda}/ \rho_\lambda(SL_n)$.

Since $SL_n$ is special, by theorem \ref{thmmain}, for any weight $\lambda$, the quotient variety $\mathcal X(n,\lambda)$ is a smooth affine homogeneous variety whose motive is mixed Tate.

\begin{remark}
    A principal $\mathfrak{sl}_2$-triple $(e,h,f)$ in a simple Lie algebra $\mathfrak g$ consists of elements satisfying the standard commutation relations
\[
[h,e] = 2e, \qquad [h,f] = -2f, \qquad [e,f] = h,
\]
with $e$ being a regular nilpotent element. By the Jacobson--Morozov theorem, see \cite{Jac} , every regular nilpotent element can be completed to such a triple, and any two principal triples are conjugate under the adjoint action of $G$, the adjoint group of $\mathfrak g$. Equivalently, if $\varphi: SL_2 \hookrightarrow G$ is the group homomorphism corresponding to the triple $(e,h,f)$, then any other principal embedding is of the form $\mathrm{Ad}_g \circ \varphi$ for some $g\in G$. Consequently, the set $\mathcal P=\mathcal P(G,SL_2)$ of all principal $SL_2$ subgroups is parametrized by the conjugacy classes
\[
\{ g \,\varphi(SL_2)\, g^{-1} \mid g \in G\} \simeq G / N_G(\varphi(SL_2)),
\]
where $N_G(\varphi(SL_2))$ denotes the normalizer of the subgroup in $G$. Therefore, the motive associated with $\mathcal P$ is mixed Tate, after inverting the order of $N_G(\varphi(SL_2))/\varphi(SL_2)$.

In the context of Drinfeld--Sokolov reduction, a fixed principal $\mathfrak{sl}_2$-triple determines the nilpotent element around which one performs the Hamiltonian or BRST reduction of the affine algebra $\widehat{\mathfrak g}$. Different choices of embeddings correspond to conjugate triples and yield isomorphic reductions locally, but globally the moduli of embeddings is captured by $G/SL_2$. This construction produces the classical and quantum $W$-algebras associated with $\mathfrak g$, which appear as the chiral symmetries of Toda field theories and as integrable hierarchies in the associated Drinfeld--Sokolov system.

\end{remark}

\subsection{The symmetric homogeneous space $X_{Sp}(p,n)$}

The symmetric space 
\[
X_{Sp}(p,n) := Sp(2n)\big/ \big(Sp(2p)\times Sp(2n-2p)\big)
\]
appears as a Higgs branch of vacua in supersymmetric gauge. In theories with gauge group $Sp(2n)$ and matter in the fundamental representation, representation, assigning vacuum expectation values (VEVs) to the scalar fields can partially break the gauge symmetry as
\[
Sp(2n)  \rightsquigarrow Sp(2p)\times Sp(2n-2p),
\] 
where $p$ is determined by the number of condensed directions allowed by the D-(and F-)term constraints \cite{IntriligatorSeiberg}. The quotient space $X_{Sp}(p,n)$ parametrizes inequivalent classical vacua and is a homogeneous, highly symmetric variety, which enables exact computations of supersymmetric indices and BPS spectra via localization techniques \cite{Witten1993}. Although $X_{Sp}(p,n)$ is an open affine scheme, it admits a canonical smooth projective compactification (the De Concini--Procesi wonderful compactification), which allows rigorous control of boundary contributions in path integrals, see also corollary \ref{CorX_Sp(p,n)isMT} below. This property guarantees integrality of BPS degeneracies and restricts the types of periods and simplifies the structure of partition functions and wall-crossing formulas \cite{NekrasovOkounkov}. Symmetric spaces of this type also appear as moduli spaces of instantons or framed quiver representations in gauge theories with symplectic gauge groups, providing a direct link between geometric representation theory and supersymmetric quantum field theory, see \cite{Nakajima1999}.

\begin{corollary}\label{CorX_Sp(p,n)isMT}
The motive $M(X_{Sp}(p,n))$ of the symmetric space 
\[
X_{Sp}(p,n) = Sp(2n)\big/ \big(Sp(2p)\times Sp(2n-2p)\big)
\]
is mixed Tate. 
\end{corollary}

\begin{proof}
    This follows from theorem \ref{thmmain} and that the group $Sp(2n)$ is special in the Grothendieck-Serre sense; see \cite[Proposition~5.10]{SGA3}.
\end{proof}

\subsection{Moduli of Lagrangian Splittings}

Let $V$ be a $2n$-dimensional vector space over a field $k$, equipped with a nondegenerate symplectic form 
\[
\omega: V \times V \longrightarrow k.
\]

\begin{definition}[Lagrangian Splitting]
A \emph{Lagrangian splitting} of $V$ is a decomposition
\[
V = L \oplus L^*,
\]
where 
\begin{itemize}
    \item $L \subset V$ is an $n$-dimensional Lagrangian subspace, i.e.\ $\omega|_L = 0$;
    \item $L^*$ is a complementary Lagrangian subspace (maximal isotropic) such that the restriction of $\omega$ induces a perfect pairing
    \[
    L \times L^* \longrightarrow k.
    \]
\end{itemize}
\end{definition}

\noindent
A point of the homogeneous space 
\[
\mathcal L Sp(n):=\frac{Sp(2n)}{GL(n)}
\]
  can be viewed either as a \emph{Lagrangian splitting} of a symplectic vector space $(V,\omega)$, that is, a decomposition $V = L \oplus L'$ into two transverse Lagrangian subspaces, or equivalently as a pair $(L,q)$ consisting of a Lagrangian subspace $L \subset V$ together with a nondegenerate symmetric bilinear form $q$ on $L$.  Indeed, given such a form $q \in \mathrm{Sym}^2(L^*)$, the graph $\Gamma_q = \{\,x + q(x)\mid x\in L\,\}$ defines a Lagrangian $L'$ transverse to $L$, and conversely every transverse $L'$ arises in this way.  This correspondence endows $Sp(2n)/GL(n)$ with the structure of an open algebraic variety parametrizing all possible real Lagrangian splittings.

\begin{corollary}
The motive $M(\mathcal L SP(n))$ is mixed Tate.
\end{corollary}

\begin{proof}
    
As an algebraic homogeneous space, one can view \(Sp(2n)\) as a \( Gl_n\)-bundle over 
\( \mathcal{L} Sp(n) \;\cong\; Sp(2n)/GL(n)\).

Now the statement follows from theorem \ref{thmmain}, corollary \ref{CorGismixedTate} and Hilbert Satz 90.
\end{proof}

\begin{remark}
    This space parametrizes all inequivalent choices of canonical coordinates up to linear transformations of the “position subspace” L. Different splittings correspond to different “polarizations” in geometric quantization.
\end{remark}

\begin{remark}
      For $k=\mathbb R$, in the context of geometric quantization, this space plays a central role: different choices of Lagrangian splittings (or equivalently of $(L,q)$) correspond to different \emph{polarizations} of the symplectic phase space, and hence to different Schrödinger realizations of the Heisenberg representation.  The metaplectic \myparagraph{(or Maslov)} structure provides the transition data between these realizations, yielding an analytic and geometric manifestation of the symplectic symmetry in quantum mechanics.
\end{remark}

\subsection{The space of nondegenerate alternating forms and Pfaffian Hypersurface}

Let $V$ be a fixed $2n$-dimensional $k$-vector space. The group $GL(V)\cong GL(2n)$ acts on the vector space $\Lambda^2 V^*$ by change of bases.

A \emph{nondegenerate alternating form} on a rank-$2n$ vector bundle $E$ over a scheme $S$ is a map
\[ \phi:\;E\otimes E\to\mathcal O_S \]
such that $\phi(v,v)=0$ for all local sections $v$ and the induced map $E\to E^\vee$ is an isomorphism.

The functor sending $S$ to the set of isomorphism classes of pairs $(E,\phi)$ as above is represented by the homogeneous space
\[ GL(2n)/Sp(2n), \]
via the orbit--stabilizer identification: fixing a reference nondegenerate alternating form $\omega_0\in \Lambda^2 V^*$, one has $Stab_{GL(V)}(\omega_0)=Sp(V,\omega_0)\cong Sp(2n)$ and the orbit of $\omega_0$ is the set of all nondegenerate alternating forms.

\begin{corollary}
    
Let $V$ be a vector space of dimension $2n$ over a field $k$ with $\operatorname{char}(k) \neq 2$.  
Consider the variety $\mathcal{B}$ whose points parametrize bilinear skew-symmetric non-degenerate forms $\beta : V \times V \to k$. The motive of $\mathcal{B}$ is mixed Tate. 

\end{corollary}

\begin{proof}
    The group $GL(V)$ acts transitively on $\mathcal{B}$ by $(g \cdot \beta)(v,w) := \beta(g^{-1}v, g^{-1}w),~ g \in GL(V), \; v,w \in V$. The stabilizer equals $Sp(V, \omega)$. Note that this group is special and thus the corollary follows from theorem \ref  {thmmain}.

\end{proof}

We describe a natural projective compactification of the symmetric space \[X=GL(2n)/Sp(2n)\], define the Pfaffian polynomial which cuts out the boundary, describe the $GL(2n)$--orbit stratification of the boundary and give explicit representatives of stabilizer subgroups. Finally we prove these stabilizers are \emph{special} groups in the sense of Grothendieck--Serre. The exposition is intended to be self-contained for readers familiar with basic algebraic groups and classical linear algebra.

\paragraph{Projective compactification and the Pfaffian hypersurface}

Let $V$ be a $2n$-dimensional vector space over a field $k$ of characteristic $\neq 2$, and fix a nondegenerate alternating form $\omega_0 \in \Lambda^2 V^*$. 
The orbit of $\omega_0$ in the vector space $\Lambda^2 V^*$ under the natural action of $GL(V)$ is
\[
GL(V) \cdot \omega_0 \;\cong\; GL(2n)/Sp(2n),
\]
which parametrizes the nonzero nondegenerate alternating $2$-forms on $V$. 
Passing to projective classes identifies forms differing by a scalar, so the stabilizer of $[\omega_0]$ in $GL(V)$ is the group of symplectic similitudes
\[
GSp(2n) = \{\, g \in GL(V) \mid g \cdot \omega_0 = \lambda(g)\, \omega_0 \text{ for some } \lambda(g)\in k^\times \,\}.
\]
Hence the open $GL(V)$-orbit in the projective space $\mathbb{P}(\Lambda^2 V^*)$ is
\[
GL(V)\cdot[\omega_0] \;\cong\; GL(2n)/GSp(2n),
\]
which is precisely the locus of projective classes $[\omega]$ of nondegenerate alternating forms. 
The variety $GL(2n)/Sp(2n)$ is the affine cone (minus the origin) over this open subset, forming a principal $\mathbb{G}_m$-bundle
\[
GL(2n)/Sp(2n) \;\longrightarrow\; GL(2n)/GSp(2n).
\]
Consequently, the projective space $\mathbb{P}(\Lambda^2 V^*) \cong \mathbb{P}^{\binom{2n}{2}-1}$ provides a natural compactification
\[
\overline{X}_{\mathrm{proj}} = \mathbb{P}(\Lambda^2 V^*), \qquad 
X = GL(2n)/GSp(2n) = \mathbb{P}(\Lambda^2 V^*) \setminus D_n,
\]
where the boundary divisor $D_n$ is the projective Pfaffian hypersurface, defined as the locus of degenerate skew forms (the vanishing of the Pfaffian polynomial of degree $n$).

\begin{remark}(Pfaffian polynomial)
Choose a generator $\mathrm{vol}\in\Lambda^{2n}V^*$ (i.e. a basis of the one-dimensional space). For any $\beta\in\Lambda^2 V^*$ the $n$-fold wedge
\[ \beta^{\wedge n}\in\Lambda^{2n}V^* \]
is a scalar multiple of $\mathrm{vol}$. Define the \emph{Pfaffian} by the identity
\[ \beta^{\wedge n} = \mathrm{Pf}(\beta)\cdot \mathrm{vol}. \]
Thus $\mathrm{Pf}$ is a homogeneous polynomial on $\Lambda^2 V^*$ of degree $n$. One checks in coordinates that for the skew-symmetric matrix $A$ of $\beta$ one has
\[ \det(A) = \big(\mathrm{Pf}(A)\big)^2. \]

The vanishing locus $D_n=V(\mathrm{Pf})\subset \mathbb P(\Lambda^2 V^*)$ is exactly the projective set of \emph{degenerate} alternating forms (those with nontrivial radical). Hence the open orbit $X$ equals the complement of the Pfaffian hypersurface.
\end{remark}

\begin{proposition}
    The motive associated with the Pfaffian hypersurface $D_n$ in $\mathbb{P}^N$ is mixed Tate.
\end{proposition}

\begin{proof}
As we have seen above the Pfaffian hypersurface $D_n$ is the complement of $GL(2n)/GSp(2n)$ in $\mathbb{P}(\Lambda^2 V^*)$, therefore the mixed Tateness of $D_n$ follows from mixed Tateness of $GL(2n)/GSp(2n)$. Since the group $GSp(2n)$ is special, the latter follows from theorem \ref{thmmain}.
\end{proof}

\paragraph{Orbit stratification of the Pfaffian}
Alternating forms on $V$ have even rank. For $0\le k\le n$ write $r=2k$ and $m=2n-r$. The subset of forms of rank exactly $r$ is a single $GL(V)$--orbit; choose a normal form representative with respect to the decomposition $V=R\oplus W$ where $\dim R=m$, $\dim W=r$,
\[ \beta_k = 0_R\oplus \omega_W, \qquad \omega_W \text{ nondegenerate on }W. \]

Concretely, in block coordinates relative to $V=R\oplus W$ the form has matrix
\[ \beta_k = \begin{pmatrix}0_m & 0\\ 0 & J_{2k}\end{pmatrix},\qquad J_{2k}=\begin{pmatrix}0&I_k\\-I_k&0\end{pmatrix}. \]

The corresponding affine orbit \( \mathcal O_{2k}\) consists of all \( \beta \) in \(\Lambda^2 V^*\) with \( rank(\beta)=2k\).


\begin{proposition}
    The motive associated with $O_{2k}$ is mixed Tate.
\end{proposition}
\begin{proof}
    
The orbit $O_{2k}$ is isomorphic to the homogeneous space $GL(2n)/H_k$, where the stabilizer subgroup $H_k$ (stabilizer of the normal form) is the block subgroup
\begin{equation}\label{eq:Hk-block}
H_k = \left\{\begin{pmatrix}a & b\\ 0 & s\end{pmatrix} : a\in GL(R),\; s\in Sp(W),\; b\in\operatorname{Hom}(W,R)\right\}.
\end{equation}

Thus $H_k$ fits into the exact sequence
\begin{equation}\label{eq:extension}
1\to U\to H_k \xrightarrow{\;\pi\;} GL(R)\times Sp(W)\to 1,
\end{equation}
where $U\cong\operatorname{Hom}(W,R)\cong\mathbb G_a^{m r}$ is the unipotent radical (upper-right blocks $b$). The map $\pi$ is the projection to block diagonal.

Since $U$ and $GL(R)\times Sp(W)$ are \emph{special} in the sense of Grothendieck--Serre, the groups $H_k$ are \emph{special} as well. Now the statement follows from theorem \ref{thmmain}.

\end{proof}

\begin{minipage}[t]{0.9\linewidth}
\myparagraph{
\noindent
\small\textbf{Esmail Arasteh Rad}, School of Mathematics, Institute for Research in Fundamental Sciences (IPM), P.O. Box: 19395-5746, Tehran, Iran 
email: earasteh@ipm.ir\\
\\[1mm]
}

\noindent
\small\textbf{Somayeh Habibi}, School of Mathematics, Institute for Research in Fundamental Sc
iences (IPM), P.O. Box: 19395-5746, Tehran, Iran
 email: \href{shabibi@ipm.ir}{shabibi@ipm.ir}

\end{minipage}

\end{document}